\documentclass[12pt]{amsart}
\usepackage[margin=1in]{geometry}
\usepackage{amsmath}
\usepackage{amsthm}
\usepackage{amsbsy}
\usepackage{amssymb}
\usepackage{hyperref}
\usepackage{amsfonts}
\usepackage{color}

\usepackage{enumitem}
\numberwithin{equation}{section}
\newtheorem{theorem}{Theorem}[section]
\newtheorem{lemma}[theorem]{Lemma}

\newtheorem{corollary}[theorem]{Corollary}

\newtheorem{hypothesis}[theorem]{Hypothesis}

\theoremstyle{remark}
\newtheorem*{remark}{Remark}

\theoremstyle{remark}

\newcommand{\Z}{\mathbb{Z}}
\newcommand{\C}{\mathbb{C}}
\newcommand{\R}{\mathbb{R}}
\newcommand{\Q}{\mathbb{Q}}
\newcommand{\D}{\mathbb{D}}
\renewcommand{\S}{\mathbb{S}^1}

\newcommand{\re}{\textup{Re}}
\newcommand{\im}{\textup{Im}}

\newcommand{\M}{\mathcal{M}}
\newcommand{\A}{\mathbb{A}}

\title[Special Values of Motivic $L$-functions]{Special Values of Motivic $L$-functions and zeta-polynomials for symmetric powers of elliptic curves}

\author{Steffen L{\"o}brich}
\address{
Mathematisches Institut, Universit\"at zu K\"oln \\
Lindenthal, Weyertal 86-90, 50931 K\"oln, Germany}
\email{ steffen.loebrich@uni-koeln.de}

\author{Wenjun Ma}
\address{
School of Mathematics, Shandong University \\
Jinan, Shandong 250100, China}
\email{wenjunma.sdu@hotmail.com}

\author{Jesse Thorner}
\address{
Department of Mathematics, Stanford University \\
Building 380, Stanford, CA 94305, United States}
\email{jthorner@stanford.edu}

\begin{document}

\begin{abstract}
Let $\M$ be a pure motive over $\Q$ of odd weight $w\geq3$, even rank $d\geq2$, and global conductor $N$ whose $L$-function $L(s,\M)$ coincides with the $L$-function of a self-dual algebraic tempered cuspidal symplectic representation of $\mathrm{GL}_{d}(\mathbb{A}_{\Q})$.  We show that a certain polynomial which generates special values of $L(s,\M)$ (including all of the critical values) has all of its zeros equidistributed on the unit circle, provided that $N$ or $w$ are sufficiently large with respect to $d$.  These special values have arithmetic significance in the context of the Bloch-Kato conjecture.  We focus on applications to symmetric powers of semistable elliptic curves over $\Q$.  Using the Rodriguez-Villegas transform, we use these results to construct large classes of ``zeta-polynomials'' (in the sense of Manin) arising from symmetric powers of semistable elliptic curves; these polynomials have a functional equation relating $s\mapsto 1-s$, and all of their zeros on the line $\Re(s)=1/2$.
\end{abstract}

\maketitle

\section{Introduction and statement of results}

Let  $f(z)=\sum_{n=1}^\infty a_f(n)q^n$ be a normalized holomorphic cuspidal modular form of even weight $k\geq2$ and level $N$, and trivial nebentypus.  Assume further that $f$ is an eigenform for the Hecke operators $T_p$ for $p\nmid N$ and $U_p$ for all $p\mid N$.  We call such a modular form a newform.  The $L$-function $L(s,f)$ associated to a newform $f$, which is given by
\begin{equation}
\label{eqn:L(s,f)}
L(s,f):=\sum_{n=1}^{\infty}\frac{a_f(n)}{n^s}=\Big(\prod_{p\mid N}\frac{1}{1-a_f(p)p^{-s}}\Big)\prod_{p\nmid N}\frac{1}{1-a_f(p)p^{-s}+p^{k-1-2s}},
\end{equation}
has an analytic continuation to $\mathbb{C}$.  The completed $L$-function
\begin{equation}
\label{eqn:functional_newform}
\Lambda(s,f)=\Big(\frac{\sqrt{N }}{2\pi}\Big)^s\Gamma(s)L(s,f),
\end{equation}
is an entire function of order one and satisfies the functional equation $\Lambda(s,f)=\varepsilon(f)\Lambda(k-s,f)$, where $\varepsilon(f)\in\{-1,1\}$.  The completed $L$-function arises as a period integral of $f$:
\begin{equation}
\label{eqn:period_int}
\Lambda(s,f)=N ^{s/2}\int_0^{\infty}f(iy)y^{s-1}dy.
\end{equation}
One defines the period polynomial associated to $f$ by $r_f(z):=\int_0^{i\infty}f(\tau)(\tau-z)^{k -2}d\tau$, which is a polynomial of degree $k -2$ in $z$.  Using \eqref{eqn:period_int}, we expand $(\tau-z)^{k -2}$ to obtain
\begin{equation}
\label{eqn:pp_newform}
r_f(z)=\Big(\frac{i}{\sqrt{N}}\Big)^{k -1}\sum_{j=0}^{k -2}{k -2\choose j}(iz\sqrt{N})^j\Lambda(k -1-j,f).
\end{equation}
By expressing $\Lambda(s,f)$ in terms of $L(s,f)$ via \eqref{eqn:functional_newform}, we see that $r_f(z)$ is a generating function for the critical values $L(1,f)$, $L(2,f),\ldots,$ $L(k -1,f)$.  For additional background and details, see \cite{JMOS} and the sources contained therein.

It follows from the functional equation for $\Lambda(s,f)$ that $r_f(z)$ satisfies a functional equation of its own, relating $r_f(\frac{z}{i\sqrt{N}})$ to $r_f(\frac{1}{iz\sqrt{N}})$ and fixing the unit circle $\S=\{z\in\C\colon|z|=1\}$.  In analogy with the expected behavior of the nontrivial zeros of the Riemann zeta function $\zeta(s)$ or the nontrivial zeros of $L(s,f)$, one might expect that all of the zeros of $r_f(\frac{z}{i\sqrt{N}})$ lie on $\S$.  Because of the similarity with the Riemann hypothesis, this has been called the {\it Riemann hypothesis for period polynomials}.  Conrey, Farmer, and Imamoglu \cite{MR3118875} proved result of this sort for the odd part of $r_f(\frac{z}{i\sqrt{N}})$, and the Riemann hypothesis for the period polynomials associated to newforms of level 1 and even weight $k\geq2$ was established by El-Guindy and Raji \cite{EGR}.  The Riemann hypothesis for period polynomials is now a theorem due to Jin, Ma, Ono, and Soundararajan \cite{JMOS} for all newforms of weight $k \geq2$ with trivial nebentypus; furthermore, they proved that if either $k $ or $N $ is sufficiently large, then the zeros of $r_f(\frac{z}{i\sqrt{N}})$ are equidistributed on $\S$.

The truth of the Riemann hypothesis for period polynomials, along with the statement of equidistribution, introduces strong conditions on the sizes of the critical values $L(1,f)$, $L(2,f),\ldots,$ $L(k -1,f)$; these values have significance in algebraic number theory and arithmetic geometry.  For newforms $f$ of weight $2$ associated to elliptic curves, $r_f(z)$ is a constant polynomial with a non-zero factor of $L(1,f)$.  If the Birch and Swinnerton-Dyer conjecture is true, then $L(1,f)$ encapsulates much of the arithmetic of the elliptic curve, including order of the Tate-Shafarevich group and whether or not the rank of the Mordell-Weil group is positive.  Unfortunately, the results in \cite{JMOS} cannot provide insight into the Birch and Swinnerton-Dyer conjecture, because for $k=2$, the period polynomial is constant.  Thus the Riemann hypothesis for period polynomials when $k=2$ is trivially satisfied without shedding light on $L(1,f)$. If $k \geq 4$, the critical values hold similar importance in the context of the Bloch-Kato conjecture \cite{BK}, which generalizes of the Birch and Swinnerton-Dyer conjecture.  

In this paper, we use the ideas in \cite{JMOS} to study critical values of motivic $L$-functions.  It is well-known that each modular $L$-function $L(s,f)$ is attached to a certain pure motive over $\mathbb{Q}$ of weight $k-1$, conductor $N$, and rank 2; furthermore, $L(s,f)$ is the $L$-function of a certain cuspidal automorphic representation of $\mathrm{GL}_2(\mathbb{A}_{\Q})$.  (Here, $\A_{\Q}$ denotes the ring of adeles of $\Q$.)  The critical values of motivic $L$-functions carry similar arithmetic significance in the context of the Bloch-Kato conjecture.  When motivic $L$-functions coincide with automorphic $L$-functions, they have important analytic properties which generalize those of $L(s,f)$.  However, there does not appear to be a canonical generating polynomial for critical values of motivic $L$-functions that generalizes the properties of $r_f(z)$.  Thus we construct a polynomial $p_{\M}(z)$ (see \eqref{stuff}) which mimics $r_f(\frac{z}{i\sqrt{N}})$ and prove the following. 

\begin{theorem}
\label{thm:main_theorem_1}
Let $\mathcal{M}$ be a pure motive over $\Q$ of odd motivic weight $w=2m+1\geq3$, even rank $d\geq 2$, global conductor $N$, and Hodge numbers $h_{\nu}$ for $0\leq\nu\leq m$ (see Section \ref{sec:motivic}).  Suppose that the $L$-function $L(s,\M)$ of $\M$ coincides with the $L$-function of an algebraic, tempered, cuspidal symplectic representation of $\mathrm{GL}_{d}(\mathbb{A}_{\Q})$.  Let $p_{\mathcal{M}}(z)$ be the polynomial defined in \eqref{stuff}.
\begin{enumerate}
\item If $m=1$ and $h_0\in\{0,1\}$, then the zeros of $p_{\mathcal{M}}(z)$ lie on $\S$ and tend to be equidistributed as $N\rightarrow\infty$.
\item If $m\geq2$, $2m^{h_m} \geq(1+1/m)^{h_0}$, and $N> A_m^d$ (where $A_m$ is defined by \eqref{eqn:A_def}), then the zeros of $p_{\mathcal{M}}(z)$ lie on $\S$ and tend to be equidistributed as $N\rightarrow\infty$.
\item If $m$ is sufficiently large, then nearly all of the zeros of $p_{\M}(z)$ lie on $\S$. (See Theorem \ref{thm:largem} for a more precise statement.)
\end{enumerate}
\end{theorem}
\begin{remark}
If $L(s,\mathcal{M})$ is the $L$-function of a newform of (modular) weight $k \geq4$, then $p_{\M}(z)$ reduces to a constant multiple of $r_f(\frac{z}{i\sqrt{N}})$, whose zeros are studied in \cite{JMOS}.
\end{remark}

It is unclear how to ensure that all of the zeros lie on $\mathbb{S}^1$ while maintaining uniformity in $d$ when $m\geq2$ and $d$ is large compared to $\log N$.  Despite this setback, we already have a result that is strong enough to address a natural family of examples, namely the odd symmetric power $L$-functions $L(s,\mathrm{Sym}^n f)$ of the newforms $f$ considered in \cite{JMOS} that do not have complex multiplication (CM). The next result follows from Theorem \ref{thm:main_theorem_1} in case of $\mathcal{M}=\mathrm{Sym}^n f$ and $n$ odd.

\begin{corollary}
\label{cor:sym^n}
	Let $n\geq3$ be an odd integer and $f$ a non-CM newform of even integral weight $k\geq2$, squarefree level $N\geq13$, trivial nebentypus, and integral Fourier coefficients. We assume that $N\ge 46$ if $(k,n)=(2,5)$ and $N\ge 17$ if $(k,n)\in\{(2,7),(4,3)\}$.     If $L(s,\mathrm{Sym}^n f)$ is the $L$-function of an algebraic tempered cuspidal symplectic representation of $\mathrm{GL}_{n+1}(\mathbb{A}_{\Q})$, then all of the zeros of $p_{\mathrm{Sym}^n f}(z)$ lie on $\mathbb{S}^1$.  The zeros tend to be equidistributed as $n$ or $N$ goes to $\infty$.
	\end{corollary}

We find the most interesting case to be when $k=2$ because the results in \cite{JMOS} are trivial in this case. By numerically checking the cases that are not covered by Corollary \ref{cor:sym^n}, we obtain the following result.

\begin{theorem}
\label{thm:sym^n}
Let $E/\Q$ be a non-CM elliptic curve of squarefree conductor $N$, and let $n\geq3$ be an odd integer.  If $L(s,\mathrm{Sym}^n E)$ is the $L$-function of an algebraic, tempered, cuspidal symplectic representation of $\mathrm{GL}_{n+1}(\mathbb{A}_{\Q})$, then all of the zeros of $p_{\mathrm{Sym}^n E}(z)$ given by \eqref{stuff} lie on $\mathbb{S}^1$. The zeros tend to be equidistributed as $n$ or $N$ goes to $\infty$.
\end{theorem}

In \cite{Manin}, Manin speculated on the existence of {\it zeta-polynomials} $Z(s)$ which (in analogy with expected behavior of the Riemann zeta function and $L(s,f)$) 
satisfy a functional equation of the form $Z(s)=\pm Z(1-s)$ and have all of their zeros lie on the line $\Re(s)=1/2$.  Furthermore, there should be a ``nice'' generating function for the sequence $\{Z(-n)\}_{n=1}^{\infty}$ along with an arithmetic-geometric interpretation of $Z(-n)$.  Manin constructed zeta-polynomials by applying the ``Rodriguez-Villegas transform'' \cite{RV} to the odd part of the period polynomial of a newform using the results in  \cite{MR3118875}; he suggests that these polynomials arise from non-Tate motives and geometric objects lying below $\mathrm{Spec}~\Z$ but not over $\mathbb{F}_1$.

Manin asked whether there exist zeta-polynomials which can be canonically constructed from the full period polynomial.  Ono, Rolen, and Sprung \cite{ORS} recently used the results in \cite{JMOS} to address this question, producing a large class of zeta-polynomials canonically constructed from the critical values of classical newforms $f$. Assuming the Bloch-Kato conjecture, these zeta-polynomials encode further Galois cohomological structure of Selmer groups for Tate-twists that have been assembled as Stirling complexes.  Moreover, in analogy with the Maclaurin expansion
\[
\frac{t}{e^t-1}=1-\frac{t}{2}+t\sum_{\ell=1}^{\infty}\zeta(-n)\cdot\frac{(-t)^\ell}{\ell!},
\]
the zeta-polynomials $Z_f(s)$ constructed in \cite{ORS} satisfy
\[
\frac{(\frac{\sqrt{N}}{i})^{k-1}r_f(\frac{z}{i\sqrt{N}})}{(1-z)^{k-1}}=\sum_{\ell=0}^{\infty}Z_f(-\ell)z^\ell.
\]

Using Theorem \ref{thm:sym^n}, we construct zeta-polynomials arising from the special values of odd symmetric power $L$-functions of semistable elliptic curves over $\Q$.  Using the Bloch-Kato conjecture, one can express the coefficients of these zeta-polynomials in terms of Tamagawa numbers and generalized Shafarevich-Tate groups of the symmetric powers.

\begin{theorem}
\label{thm:zeta}
Let $E/\Q$ be a non-CM elliptic curve, and let $n\geq3$ be odd.  Suppose that $L(s,\mathrm{Sym}^n E)$ is the $L$-function of an algebraic, tempered, cuspidal symplectic representation of $\mathrm{GL}_{n+1}(\mathbb{A}_{\Q})$.  Let $Z_{\mathrm{Sym}^n E}(s)$ be the polynomial defined by \eqref{eqn:Z}.  The following are true.
\begin{enumerate}
	\item For all $s\in\C$, we have that $Z_{\mathrm{Sym}^n E}(s)=\varepsilon(\mathrm{Sym}^n E)Z_{\mathrm{Sym}^n E}(1-s)$, where $\varepsilon(\mathrm{Sym}^n E)$ is the sign of the functional equation for $L(s,\mathrm{Sym}^n E)$.
	\item If $Z_{\mathrm{Sym}^n E}(\rho)=0$, then $\Re(\rho)=1/2$.
	\item We have the Maclaurin expansion
\[
\frac{p_{\mathrm{Sym}^n E}(z)}{(1-z)^n}=\sum_{\ell=0}^{\infty}Z_{\mathrm{Sym}^n E}(-\ell)z^{\ell}.
\]
\end{enumerate}
\end{theorem}

We review motivic $L$-functions and their conjectured analytic properties in Section \ref{sec:motivic}.  In Section \ref{sec:prelims}, we prove some lemmas that are needed for the proofs of Theorem \ref{thm:main_theorem_1}, which we prove in Sections \ref{sec:proofs_1} and \ref{sec:large_m}.  We then discuss symmetric power $L$-functions and prove Theorems \ref{thm:sym^n} and \ref{thm:zeta} in Sections \ref{sec:symmetric_powers} and \ref{sec:zeta}.

\subsection*{Acknowledgements}

This project was suggested by Ken Ono; the authors thank him for his comments and support.  The authors also thank Seokho Jin, Robert Lemke Oliver, Akshay Venkatesh, and Don Zagier for helpful conversations.  S.L. thanks the Fulbright Commission for their generous support and the Department of Mathematics and Computer Science at Emory University for their hospitality.  W.M. thanks the Chinese Scholarship Council for its generous support.  J.T. is supported by a NSF Mathematical Sciences Postdoctoral Research Fellowship.

\section{Motivic $L$-functions}
\label{sec:motivic}

We begin by recalling the conjectural properties of motivic $L$-functions.  For more details, see Serre \cite{Ser} and Iwaniec and Kowalski \cite[Chapter 5]{IK}.

\subsection{Conjectured analytic properties}

Define a pure motive $\mathcal{M}$ over $\Q$ of weight $w $, rank $d$, and global conductor $N $ by specifying Betti, de Rham, and $\ell$-adic realizations (for each prime $\ell$)
\[
H_B(\M),\quad H_{dR}(\M),\quad H_{\ell}(\M)
\]
which are vector spaces of dimension $d$ over $\Q$, $\Q$, and $\Q_{\ell}$, respectively; each is endowed with additional structures and comparison isomorphisms as in \cite{333,555}.  In particular, $H_B(\M)$ admits an involution $\rho_B$, $H_{\ell}(\M)$ is a $\mathrm{Gal}(\bar{\Q}/\Q)$-module, and there is a Hodge decomposition into $\mathbb{C}$-vector spaces
\[
H_B(\M)\otimes\mathbb{C}=\bigoplus_{\substack{i+j=w  \\ i,j\geq0}}H^{i,j}(\M).
\]
The involution $\rho_B$ acts on $H^{i,j}(\M)$ by $\rho_B(H^{i,j}(\M))=H^{j,i}(\M)$.  When $w $ is even, this tells us that $H^{w /2,w /2}(\M)$ is invariant under $\rho_B$; when $w $ is odd, we take $H^{w /2,w /2}(\M)=\{0\}$.  If $w $ is even and $H^{w /2,w /2}(\M)\neq\{0\}$, then the involution $\rho_B$ acts on $H^{w /2,w /2}(\M)$ by $\alpha\in\{-1,1\}$; we then define the quantity $b^{\pm}(\M)$ by
\[
b^{\alpha}(\M):=\dim_{\C}\{x\in H^{w /2,w /2}(\M):\rho_B(x)=\alpha(-1)^{w /2}x\},\qquad \alpha\in\{-1,1\}.
\]
We denote by $\rho_{\ell}$ the representation which induces the $\mathrm{Gal}(\bar{\Q}/\Q)$-module structure on $H_{\ell}(\M)$.

For any prime $p$, let $\mathrm{Frob}_p\in\mathrm{Gal}(\bar{\Q}/\Q)$ be the Frobenius element at $p$, which is defined modulo conjugation and modulo the inertia subgroup $I_p\subset G_p\subset \mathrm{Gal}(\bar{\Q}/\Q)$ of the decomposition group $G_p$.  Define
\[
L_{\ell,p}(X,\M):=\det(1-X\cdot \rho_{\ell}(\mathrm{Frob}_p^{-1})|_{H_{\ell}(\M)^{I_p}})^{-1}=\prod_{j=1}^{d }(1-\alpha_{\M}(j,\ell,p)X)^{-1}.
\]
One typically assumes (and expects) that $L_{\ell,p}(X,\M)$ and $\alpha_{\M}(j,\ell,p)$ are in fact independent of $\ell$; as such, we write $L_{p}(X,\M)$ and $\alpha_{\M}(j,p)$ instead of $L_{\ell,p}(X,\M)$ and $\alpha_{\M}(j,\ell,p)$ for convenience.  (If this is not true, our results are only affected notationally.)  The Euler product and Dirichlet series representations of $L(s,\M)$ are now given as
\[
L(s,\M):=\prod_{p}L_p(p^{-s},\M)=:\sum_{n\geq 1}\frac{\lambda_\M(n)}{n^s}
\]
with $\lambda_\M(n)\in\C$. Both the Euler product and the Dirichlet series converge absolutely in the half-plane $\re(s)>w /2+1$.

Define the $\nu$-th Hodge number of $\M$ by $h_{\nu}:=\dim_{\mathbb{C}}H^{\nu,w -\nu}(\M)$.  Let $\Gamma_{\R}(s)=\pi^{-s/2}\Gamma(s/2)$ and $\Gamma_{\C}(s)=2(2\pi)^{-s}\Gamma(s)$, and define
\[
L_{\infty}(s,\M)=\Gamma_{\R}(s-w/2)^{b^+(\M)}\Gamma_{\R}(s+1-w/2)^{b^-(\M)}\prod_{0\leq \nu<w/2} \Gamma_{\C}(s-\nu)^{h_{\nu}}.
\]
Because we consider $\M$ over $\Q$, the degree of $L(s,\M)$ also equals
\begin{equation}
\label{eqn:d_def}
d=b^{+}(\M)+b^{-}(\M)+2\sum_{0\leq\nu< w /2} h_{\nu}.
\end{equation}

We now describe the hypotheses for $L(s,\M)$ which are crucial to our arguments.

\begin{hypothesis}
\label{hyp}
Let $\M$ be a self-dual motive of weight $w\geq1$, rank $d\geq1$, and global conductor $N$.  Let $L(s,\M)$ be the $L$-function of $\M$.  The following are true.
\begin{enumerate}
\item Self-duality: For all $n\geq 1$, we have that $\lambda_\M (n)\in\R$.
\item The generalized Ramanujan conjecture (GRC): We have that $|\lambda_{\M}(n)|\leq d(n)n^{w /2}$ for every $n\geq1$, where $d(n)$ is the usual divisor function.
\item Analytic continuation:  The function $\Lambda(s,\M):=N ^{s/2}L_{\infty}(s,\M)L(s,\M)$ is entire of order 1.
\item Functional equation:  There exists $\varepsilon(\M)\in\{-1,1\}$ such that for every $s\in\mathbb{C}$, we have that $\Lambda(s,\M)=\varepsilon(\M)\Lambda(w +1-s,\M)$.  We call $\varepsilon(\M)$ the root number of $\M$.
\item We have $\Lambda(\frac{w +1}{2},\M)\geq0$. 
\end{enumerate}
\end{hypothesis}

Property 5 follows from the Generalized Riemann Hypothesis for $L(s,\M)$, and it is known unconditionally in many cases.  Every other property of Hypothesis \ref{hyp} is immediately satisfied when $L(s,\M)$ coincides with the $L$-function $L(s,\pi_{\M})$ of an algebraic, self-dual, tempered, cuspidal automorphic representation $\pi_{\M}$ of $\mathrm{GL}_{d}(\A_{\Q})$, where $d$ is the rank of $\mathcal{M}$.   This is predicted by the Langlands program but is known unconditionally for a small (though highly important and useful) collection of motivic $L$-functions, such as the $L$-functions associated to newforms.  In what follows, we will always assume that $L(s,\M)=L(s,\pi_{\M})$ for some $\pi_{\M}$ in $\mathcal{A}_d(\Q)$, the set of all algebraic, self-dual, tempered, cuspidal automorphic representations of $\mathrm{GL}_d(\mathbb{A}_{\Q})$, where $d$ is the rank of $\mathcal{M}$.

\subsection{Critical values and Hodge numbers}

Following Deligne \cite{555}, we define an integer $n$ to be {\it critical} for $\M$ if neither $L_{\infty}(s,\M)$ nor $L_{\infty}(w+1-s,\M)$ has a pole at $s=n$; if $n$ is critical for $\M$, then we call $L(n,\M)$ a {\it critical value} of $L(s,\M)$.  With this definition, the critical integers are purely dictated by the Hodge numbers.  The simplest situation occurs when $b^+(\M)$ and $b^-(\M)$ both equal zero; then the set of integers $n$ which are critical for $\M$ are precisely those which lie in the interval
\begin{equation}
\label{eqn:def_critical_interval}
\Big(\max_{\substack{h_{\nu}\neq 0 \\ 0\leq\nu<w /2}}\nu,w -\max_{\substack{h_{\nu}\neq 0 \\ 0\leq\nu<w /2}}\nu\Big].
\end{equation}
(When $\M$ corresponds with a newform $f$ of (modular) weight $k $, then $w =k -1$, $h_{0}=1$, and $h_{\nu}=0$ for all $1\leq\nu<\frac{k -1}{2}$.  Thus the critical values of $L(s,f)$ are $L(n,f)$ for integers $1\leq n\leq k -1$.)  On the other hand, if at least one of $b^+(\M)$ and $b^-(\M)$ is nonzero, then the distribution of critical integers is slightly more complicated.  Briefly stated, if just one of $b^+(\M)$ and $b^-(\M)$ are nonzero, then the critical integers of $\M$ will not be consecutive integers; if both $b^+(\M)$ and $b^-(\M)$ are nonzero, then $L(s,\M)$ has no critical values.  For simplicity, we only consider motives $\M$ such that $w $ is odd and $h_{\nu}\geq1$ for some $0\leq\nu<w /2$.  Thus $b^+(\M)=b^-(\M)=0$, the integers that are critical for $\M$ are symmetric about the critical line for $L(s,\M)$, and $d\geq2$.  We will study polynomials that generate the values $L(1,\M)$, $L(2,\M),\ldots,L(w ,\M)$, which, by our hypotheses, includes all of the critical values.

When $w$ is odd, we see that $d$ must be even (see \eqref{eqn:d_def}).  Now, consider now the exterior square representation $\mathrm{Ext}^2(\pi_{\M})$ and the Euler product
\[
L(s,\mathrm{Ext}^2(\pi_{\M}))=\prod_p L_p(p^{-s},\mathrm{Ext}^2(\pi_{\M})),
\]
where at each prime $p\nmid N$ we have
\begin{equation}
\label{eqn:ext}
L_p(p^{-s},\mathrm{Ext}^2(\M))=\prod_{1\leq j<k\leq n}(1-\alpha_{\M}(j,p)\alpha_{\M}(k,p)p^{-s})^{-1}.
\end{equation}
We know that $L(s,\mathrm{Ext}^2(\pi_{\M}))$ has a meromorphic continuation to $\mathbb{C}$ with no poles outside of the set $\{\frac{w'}{2},\frac{w'}{2}+1\}$, where $w'$ is the weight of $\mathrm{Ext}^2(\pi_{\M})$ \cite{MS}.  If $L(s,\mathrm{Ext}^2(\pi_{\M}))$ has a pole at $s=\frac{w}{2}+1$, then $\pi_{\M}$ is a cuspidal symplectic representation of $\mathrm{GL}_{d }(\A_{\Q})$; let $\mathcal{A}_d^{\mathrm{s}}(\Q)$ denote the set of such representations.  For any $\pi_{\M}\in\mathcal{A}_d^{\mathrm{s}}(\Q)$, Lapid and Rallis \cite{LR} proved that $\Lambda(\frac{w +1}{2},\pi_{\M})\geq0$.  (This vastly generalizes a result of Waldsuprger \cite{Waldspurger} for $L$-functions of newforms.)  Therefore, the hypotheses of Theorem \ref{thm:main_theorem_1} succinctly describe the most natural class of motivic $L$-functions for which the methods in \cite{JMOS} can be used for studying special and critical values.

In Theorem \ref{thm:main_theorem_1}, we require that $2m^{h_m}\geq(1+1/m)^{h_0}$.  This not true of all $\mathcal{M}$.  In fact, for {\it any} integer $m\geq0$ and {\it any} collection of nonnegative integers $h_0,\ldots, h_{m}$, there exists a motive of weight $2m+1$ with Hodge numbers $h_0,\ldots, h_m$; see Arapura \cite{Arapura} and Schreieder \cite{Schreieder} for explicit constructions.  However, for newforms and their symmetric powers (see Section \ref{sec:symmetric_powers}) as well as many other interesting cases, we have $h_{\nu}\in\{0,1\}$ for each $1\leq \nu\leq m$.

\section{Preliminary Lemmas and Setup}
\label{sec:prelims}

Let $\M$ be a pure motive over $\Q$ of rank $d\geq 2$ with global conductor $N $, odd weight $w=2m+1\geq3$, root number $\varepsilon=\varepsilon(\M)$, and Hodge numbers $h_{\nu}$ for $0\leq\nu\leq m$.  (It will be more notationally convenient for us to use $m$ instead of $w$.)  For convenience, we let $\S:=\{z\in\C:|z|=1\}$ and $\D:=\{z\in\C:|z|<1\}$.

We now define our first analogue of \eqref{eqn:pp_newform} by letting
\begin{equation}
\label{stuff}
p_{\M}(z):=\sum_{j=0}^{2m} \Big[\prod_{\nu=0}^m {2m-\nu\choose m-|m-j|}^{h_{\nu}}\Big] \Lambda(2m+1-j,\M)z^{j}.
\end{equation}
  Using the functional equation of $\Lambda(s,\M)$ in Part (3) of Hypothesis \ref{hyp}, we have that
\begin{equation}
\label{eqn:p_def_func}
p_{\M}(z)=\varepsilon z^{m}(P_{\M}(z)+\varepsilon P_{\M}(1/z)),
\end{equation}
where
\[
P_{\M}(z):=\frac{1}{2}\Big[\prod_{\nu=0}^m{2m-\nu\choose m}^{h_{\nu}}\Big]\Lambda(m+1,\M)+\sum_{j=1}^{m}\Big[\prod_{\nu=0}^m{2m-\nu\choose m-j}^{h_{\nu}}\Big]\Lambda(m+1+j,\M)z^j.
\]
If $z=e^{i\theta}\in\S$, then $P_{\M}(z)+\varepsilon P_{\M}(1/z)$ is a trigonometric polynomial in either $\cos(\theta)$ or $\sin(\theta)$ (depending on the sign of $\varepsilon$).  Therefore, to prove that the zeros of $p_{\M}(z)$ are equidistributed on $\S$, we find the correct number and placement of sign changes of $P_{\M}(z)+\varepsilon P_{\M}(1/z)$ as $\theta$ varies along $[0,2\pi)$.

Since $\Lambda(s,\M)$ is an entire function of order one, there exist constants $A=A_{\M}$ and $B=B_{\M}$ such that $\Lambda(s,\M)$ has the Hadamard factorization
\begin{equation}
\label{eqn:hadamard2}
\Lambda (s,\M)=e^{A+Bs}\prod_{\rho}\Big(1-\frac{s}{\rho}\Big)e^{s/\rho},
\end{equation}
where the product runs over the zeros $\rho$ of $\Lambda(s,\M)$.  Self-duality and the functional equation of $\Lambda(s,\M)$ imply that if $\rho$ is a zero of $\Lambda(s,\M)$, then so are $\bar{\rho}$ and $w+1-\rho$.  Self-duality also implies that $\Lambda(s,\M)$ is real-valued on the real line, and in view of the functional equation of $\Lambda(s,\M)$, we have that $B$ is real-valued and $B=-\sum_{\rho}\re(\rho^{-1})=-\sum_{\rho}\re(\rho)|\rho|^{-2}$.  Thus if $s\in\R$, then
\begin{equation}
\label{eqn:hadamard}
\Lambda(s,\M)=e^{A}\Big[\prod_{\rho\in\R}\Big(1-\frac{s}{\rho}\Big)\Big]\cdot\Big[\prod_{\im(\rho)>0}\Big|1-\frac{s}{\rho}\Big|^2\Big].
\end{equation}

\begin{lemma}
\label{lem:monotonicity}
The function $\Lambda(s,\M)$ is monotonically increasing for $s\geq m+3/2$; moreover,
\[
0\leq\Lambda(m+1,\M)\leq \Lambda(m+2,\M)\leq \Lambda(m+3,\M)\leq \Lambda(m+4,\M) \leq \dots
\]
If $\varepsilon =-1$, then $\Lambda(m+1,\M)=0$ and 
\[
0\leq\Lambda(m+2,\M)\leq \frac{1}{2}\Lambda(m+3,\M) \leq\frac{1}{3}\Lambda(m+4,\M)\leq \dots
\]
\end{lemma}
\begin{proof}
All of the zeros in the product \eqref{eqn:hadamard} lie in the vertical strip $|m+1-\re(s)|<1/2$, and we see that $|1-s/\rho|$ is increasing for $s\geq m+3/2$.  Thus by \eqref{eqn:hadamard}, we have that $\Lambda(s,\M)$ is increasing for $s\geq m+3/2$.  Moreover, $|1-\frac{m+1}{\rho}|\leq|1-\frac{m+2}{\rho}|$, so $\Lambda(m+1,\M)\leq\Lambda(m+2,\M)$.  When $\varepsilon=-1$, we apply the same reasoning and take into account that $\Lambda(s,\M)$ has a zero of odd order at $s=m+1$.
\end{proof}

\begin{lemma}
\label{zetaest}
For $0<a<b$, we have 
\[
\frac{L(m+3/2+a,\M)}{L(m+3/2+b,\M)}\leq \Big(\frac{\zeta(1+a)}{\zeta(1+b)}\Big)^{d},
\]
where $\zeta(s)$ is the Riemann zeta function.
\end{lemma}

\begin{proof}
The Euler product for $L(s,\M)$ gives rise to the function $\Lambda_{\M}(n)$ which is defined by the Dirichlet series identity
\[
-\frac{L'}{L}(s,\M)=\sum_{n=1}^{\infty}\frac{\Lambda_{\M}(n)}{n^s}.
\]
One sees that $|\Lambda_{\M}(n)|\leq d n^{w/2}\Lambda(n)$ for all $n\geq1$, where $\Lambda(n)$ is the usual von Mangoldt function; this estimate follows from Part (4) of Hypothesis \ref{hyp}.

Let $0<a\leq t\leq b$.  By the above discussion,
\[
\Big|-\frac{L'}{L}(m+3/2+t,\M)\Big|\leq \sum_{n= 1}^{\infty}\Big|\frac{\Lambda_{\M}(n)}{n^{1+t+w/2}}\Big|\leq d\sum_{n=1}^{\infty}\frac{\Lambda(n)}{n^{1+t}}=-d\frac{\zeta'}{\zeta}(1+t).
\]
Consequently, 
\[
\frac{L(m+3/2+a,\M)}{L(m+3/2+b,\M)} = \exp\Big(\int_a ^b -\frac{L'}{L}(m+3/2+t,\M)dt\Big) \leq \exp\Big(-d\int_a ^b \frac{\zeta'}{\zeta}(1+t)dt\Big),
\]
which equals the right hand side of the desired inequality.
\end{proof}

We will also use the following lemma due to P\'olya \cite{Polya} and Szeg\"o \cite{Szego} on the zeros of trigonometric polynomials.

\begin{lemma}
\label{sze}
If $0\leq a_0 \leq a_1 \leq \dots \leq a_{n-1} < a_n$, then the polynomial $\sum_{j=0}^n a_n\cos(n\theta)$ has exactly one zero in each interval $(\frac{2j-1}{2n+1}\pi, \frac{2j+1}{2n+1}\pi)$ for $1\leq j\leq n$.  Also, the polynomial $\sum_{j=1}^n a_n \sin(n\theta)$ has a zero at $\theta=0$ and exactly one zero in each interval $(\frac{2j}{2n+1}\pi, \frac{2(j+1)}{2n+1}\pi)$ for $1\leq j\leq n-1$.
\end{lemma}

\section{Proof of Theorem \ref{thm:main_theorem_1} when $N$ is large}
\label{sec:proofs_1}

Our proof of Theorem \ref{thm:main_theorem_1} is broken into two cases.  First we consider the case when $m=1$, in which case $P_{\M}(z)$ is linear.  Then we consider the case where $m\geq2$.

\subsection{Case 1:  $m=1$}
We have $P_{\M}(z)=\Lambda(3,\M)z + 2^{h_0-1}\Lambda(2,\M)$.  If $\varepsilon=-1$, then
\[
p_{\M}(z)=z^m(P_{\M}(z)+\varepsilon P_{\M}(1/z))=(z^2-1)\Lambda(3,\M).
\]
Since $-1$ and $1$ are the roots and they are clearly equidistributed on $\S$, Theorem \ref{thm:main_theorem_1}  is proven for all $d$ and all $N$.

On the other hand, if $\varepsilon=1$ and $z=e^{i\theta}$ for some $\theta\in[0,2\pi)$, then
\begin{equation}
\label{eqn:m=1_even}
z^m(P_{\M}(z)+\varepsilon P_{\M}(1/z))=2e^{i\theta}(\cos(\theta)\Lambda(3,\M)+2^{h_0-1}\Lambda(2,\M)).
\end{equation}
Since $\Lambda(2,\M)<\Lambda(3,\M)$ by Lemma \ref{lem:monotonicity}, \eqref{eqn:m=1_even} has two roots for $\theta\in[0,2\pi)$; these are the two values of $\theta$ for which $\cos\theta=-2^{h_0-1}\Lambda(2,\M)/\Lambda(3,\M)$, provided that $h_0\in\{0, 1\}$.  This places the roots of $p_{\M}(z)$ on $\mathbb{S}^1$.

We now show that the zeros of \eqref{eqn:m=1_even} are equidistributed when $N$ is large.  By the definition of $\Lambda(s,\M)$ and Lemma \ref{lem:monotonicity}, we have that $\Lambda(3,\M)\gg N^{3/2}$, whereas
\[
\Lambda(2,\M)\leq\sup_{t\in\R}|\Lambda(5/2+\epsilon+it,\M)|\ll N^{5/4+\epsilon}
\]
for any $\epsilon>0$.  (This uses the Phragm\'en-Lindel\"{o}f convexity bound for $L(s,\M)$ in the critical strip is given by \cite[Equation 5.21]{IK}.)  Therefore, $\Lambda(2,\M)/\Lambda(3,\M)\ll N^{-1/4+\epsilon}$, and so the corresponding values of $\theta$ tend to $\pi/2$ and $3\pi/2$.  Thus if $\varepsilon=1$, then the zeros of $p_{\M}(z)$ are $\pm i+O(N^{-1/4+\epsilon})$.

\subsection{Case 2: $m\geq2$}

We will show that if $N$ is sufficiently large and $2m^{h_m} \geq (1+1/m)^{h_0}$, then the zeros of $p_{\M}(z)$ are equidistributed on $\S$.  This follows as soon as we show that we can apply Lemma \ref{sze} to the real and imaginary parts of $P_{\M}(e^{i\theta})+\varepsilon P_{\M}(e^{-i\theta}) $.  So that we may apply Lemma \ref{sze}, we will verify that 
\[
\Big[\prod_{\nu=0}^m{2m-\nu\choose m-j}^{h_{\nu}}\Big]\Lambda(m+1+j,\M)<  \Big[\prod_{\nu=0}^m{2m-\nu\choose m-(j+1)}^{h_{\nu}}\Big]\Lambda(m+2+j,\M)
\]
for all $1\leq j\leq m-1$ and 
\[
\frac{1}{2}\Big[\prod_{\nu=0}^m{2m-\nu\choose m}^{h_{\nu}}\Big]\Lambda(m+1,\M)\leq \Big[\prod_{\nu=0}^m{2m-\nu\choose m-1}^{h_{\nu}}\Big]\Lambda(m+2,\M).
\]
By the definitions of $\Lambda(s,\M)$ and $d$, this is equivalent to
\begin{equation}
\label{eqn:11}
\frac{1}{(m-j)^{d/2}}L(m+j+1,\M) < \Big(\frac{N}{(2\pi)^{d}}\Big)^{1/2}L(m+j+2,\M)
\end{equation}
for each $1\leq j\leq m-1$ and
\begin{equation}
\label{eqn:22}
\frac{1}{2}\Big[\prod_{\nu=0}^m \frac{1}{m^{h_\nu}}\Big]\Lambda(m+1,\M) \leq \Big[\prod_{\nu=0}^m \frac{1}{(m+1-\nu)^{h_\nu}}\Big]\Lambda(m+2,\M).
\end{equation}

By Lemma \ref{zetaest} we have
\[
\frac{L(m+j+1,\M)}{L(m+j+2,\M)}\leq \Big(\frac{\zeta(j+1/2)}{\zeta(j+3/2)}\Big)^{d}.
\]
Therefore, \eqref{eqn:11} is satisfied when $N> A_m^d$, where
\begin{equation}
\label{eqn:A_def}
A_m :=\max_{1\leq j\leq m-1}\frac{2\pi }{m-j}\cdot\Big(\frac{\zeta(j+1/2)}{\zeta(j+3/2)}\Big)^2.
\end{equation}
Since $\Lambda(m+1,\M) \leq \Lambda(m+2,\M)$, \eqref{eqn:22} is satisfied when $2m^{h_m} \geq (1+1/m)^{h_0}$, as can be seen using term-by-term comparison.  This completes the proof.

It is straightforward to compute $A_2\leq 23.83$, $A_3\leq 11.92$, $A_m\leq 8$ for $m\geq4$, and $\lim_{m\to\infty} A_m=2\pi$.  Thus the above proof cannot produce a lower bound for $N$ better than $(2\pi)^d$; we must handle the cases where $N\leq A_m^d$ differently.

\section{Proof of Theorem \ref{thm:main_theorem_1} when $m$ is large}
\label{sec:large_m}

On the unit circle, $r_f(z)$ is well-approximated by an exponential function \cite[Section 6]{JMOS}, but if $\M$ is arbitrary, then $p_{\M}(z)$ is well-approximated on the unit circle by a certain generalized hypergeometric function. Unfortunately, it is computationally intractable to locate the zeros of the real and imaginary parts of generalized hypergeometric functions, and Rouch\'e's Theorem only gives us the zeros of the real and imaginary part simultaneously. Therefore, we can only prove that ``most'' zeros (depending on $d$ and $N$) lie on the unit circle as the weight becomes large.

Let $d$ be fixed.  If we define
\begin{equation}
\label{qdef}
Q_{\M} (z):=z^m\sum_{j=0}^{m-1}\frac{1}{(j!)^{\frac{d}{2}}}\frac{(2\pi)^{\frac{dj}{2}}}{(\sqrt{N}z)^j}\frac{L(2m+1-j,\M)}{L(2m+1,\M)}+\frac{1}{2(m!)^{d/2}}\Big(\frac{(2\pi)^{d/2}}{\sqrt{N}}\Big)^{m}\frac{L(m+1,\M)}{L(2m+1,\M)},
\end{equation}
then we may write $P_{\M}(z)$ as
\begin{equation}
\label{eqn:PM_def}
P_{\M} (z)=\Big[\prod_{\nu=0}^{m}((2m-\nu)!)^{h_\nu}\Big]\Big(\frac{\sqrt{N}}{(2\pi)^{d/2}}\Big)^{2m+1}L(2m+1,\M)Q_{\M}(z).
\end{equation}
Define
\begin{equation}
\label{fdndef}
F_{d,N}(z):=\sum_{j=0}^{\infty} \frac{1}{(j!)^{d/2}}\Big(\frac{(2\pi)^{d/2}}{\sqrt{N}}z\Big)^j,
\end{equation}
which we approximate by its partial sums $T_{m,d,N} (z):= \sum_{j= 0}^m \frac{1}{(j!)^{d/2}}(\frac{(2\pi)^{d/2}}{\sqrt{N}}z)^j$.

Now we decompose $Q_{\M}(z)$ into the sum
\begin{equation}
\label{eqn:Q_decomp}
Q_{\M} (z)= z^m T_{m,d,N}(1/z) + S (z)+ \frac{1}{2(m!)^{d/2}}\Big(\frac{(2\pi)^{d/2}}{\sqrt{N}}\Big)^{m}\frac{L(m+1,\M)}{L(2m+1,\M)}
\end{equation}
with
$$
S(z):= z^m\sum_{j=0}^{m-1}\frac{1}{(j!)^{d/2}}\Big(\frac{(2\pi)^{d/2}}{\sqrt{N}z}\Big)^j\Big(\frac{L(2m+1-j,\M)}{L(2m+1,\M)}-1\Big).
$$
It follows from \cite[Theorem 2.2]{EGR} that $p_\M(z)$ has as many zeros on $\S$ as $Q_{\M} (z)$ has inside $\D$. Thus Part 3 of Theorem \ref{thm:main_theorem_1} follows from the following statement.

\begin{theorem}
\label{thm:largem}
Let $c_{d,N}$ denote the number of zeros of $F_{d,N}(z)$ inside $\D$. If $m$ is sufficiently large, then $Q_{\M}(z)$ has $m-c_{d,N}$ zeros inside $\D$.
\end{theorem}

\begin{proof}
We use Rouch\'e's Theorem. First, for $|z|=1$, we estimate with Lemma \ref{zetaest}

\begin{align*}
|S (z)| &\leq \sum_{j=0}^{m-1}\frac{1}{(j!)^{d/2}}\Big(\frac{(2\pi)^{d/2}}{\sqrt{N}}\Big)^j \Big(\frac{L(2m+1-j,\M)}{L(2m+1,\M)}-1\Big) \\
&\leq \sum_{j=0}^{m-1}\frac{1}{(j!)^{d/2}}\Big(\frac{(2\pi)^{d/2}}{\sqrt{N}}\Big)^j \Big(\zeta\Big( m+ \frac{1}{2}-j\Big)^{d} -1\Big)
\end{align*}
The function $x\mapsto 2^x (\zeta(\frac{1}{2}  + x)^{d} -1)$ is monotonically decreasing for $x\geq 1$, so
\begin{equation}
\label{eqn:Sz_bound_1}
|S (z)| \leq  \sum_{j=0}^{m-1}\frac{4}{(j!)^{d/2}}\Big(\frac{(2\pi)^{d/2}}{\sqrt{N}}\Big)^j 2^{j-m}(\zeta(3/2)^{d} -1) < 2^{2-m}(\zeta(3/2)^{d} -1) F_{d,N}(2).
\end{equation}
Furthermore,
\begin{equation}
\label{eqn:Sz_bound_2}
\frac{1}{2(m!)^{d/2}}\Big(\frac{(2\pi)^{d/2}}{\sqrt{N}}\Big)^{m}\frac{L(m+1,\M)}{L(2m+1,\M)}\ll \frac{1}{(m!)^{d/2}}\Big(\frac{(2\pi)^{d/2}}{\sqrt{N}}\Big)^{m}.
\end{equation}
If $d$ is fixed, then both \eqref{eqn:Sz_bound_1} and \eqref{eqn:Sz_bound_2} can be made arbitrarily small if $m$ is sufficiently large. \\

We first assume that $F_{d,N}$ has no zeros on $\S$. Since $T_{m,d,N}(z)$ converges to $F_{d,N}(z)$ locally uniformly as $m$ tends to infinity, we have 
$$
\min_{z\in\S}\left|z^m T_{m,d,N}(1/z)\right| = \min_{z\in\S}\left|T_{m,d,N}(z)\right| > \frac{1}{2}\min_{z\in\S}\left|F_{d,N}(z)\right|
$$
for $m$ large enough. We conclude for these $m$, the functions $Q_{\M}(z)$ and $z^m T_{m,d,N}(1/z)$ have the same number of zeros inside $\D$ by Rouch\'e's Theorem.  Every zero of $z^m T_{m,d,N}(1/z)$ inside $\D$ is the inverse of a zero of $T_{m,d,N}(z)$ outside $\D$. Again using locally uniform convergence, we see that, if $m$ is sufficiently large, then $F_{d,N}(z)$ and $T_{m,d,N}(z)$ have the same number of zeros inside $\D$, namely $c_{d,N}$. This implies that $z^m T_{m,d,N}(1/z)$, and hence $Q_{\M}(z)$, has $m-c_{d,N}$ zeros inside $\D$.

If $F_{d,N}$ has zeros on $\S$, then we choose an $r>1$, such that all the zeros of $F_{d,N}$ in the region $\{ r^{-1}\le |z| \le r\}$ lie on $\S$ and slightly modify the argument above by applying Rouch\'e's Theorem to the circle $\{|z|=r\}$.
\end{proof}

By taking $d=2$, we have that $F_{2,N}(z)=\exp(\frac{2\pi}{\sqrt{N}}z)$.  Since $F_{2,N}(z)$ has no zeros in $\D$, we have that $c_{d,N}=0$; thus $p_{\M}(z)$ has all of its zeros on $\mathbb{S}^1$, as shown in \cite{JMOS}.  However, for $d=4$, the situation already becomes noticeably more complicated; when $d=4$, we have that $F_{4,N}(z)=I_0(4\pi N^{-1/4}\sqrt{z})$, where $I_0$ denotes the $I$-Bessel function.  When $d\geq 6$, $F_{d,N}(z)$ is a generalized hypergeometric function.  To illustrate the difficulty when $d\geq4$, we directly compute
\[
c_{4,N}=\begin{cases}
4&\mbox{if $N=1$},\\
3&\mbox{if $2\leq N\leq 4$},\\
2&\mbox{if $5\leq N\leq 26$},\\
1&\mbox{if $27\leq N\leq 745$},\\
0&\mbox{if $746\leq N$},
\end{cases}
\qquad
c_{6,N}=\begin{cases}
	5&\mbox{if $N=1$},\\
	4&\mbox{if $2\leq N\leq 6$},\\
	3&\mbox{if $7\leq N\leq 37$},\\
	2&\mbox{if $38\leq N\leq 494$},\\
	1&\mbox{if $495\leq N\leq 45606$},\\
	0&\mbox{if $45607\leq N$}.
\end{cases}
\]
To see how these compare with those of the previous section, we observe that $746\approx\frac{1}{2}(2\pi)^4$ and $45607\approx\frac{3}{4}(2\pi)^6$.  Thus it appears that the weight aspect of the results in \cite{JMOS} do not readily generalize to our setting when $d$ is large.

\section{Symmetric Power $L$-functions and the Proof of Theorem \ref{thm:sym^n}}
\label{sec:symmetric_powers}

\subsection{Symmetric power $L$-functions of non-CM newforms}

Let $f$ be a non-CM newform of even weight $k \geq2$, squarefree level $N$, and trivial nebentypus.  It is well-known that $L(s,f)$ is a motivic $L$-function satisfying Hypothesis \ref{hyp} with weight $w=k -1$, rank $d=2$, and global conductor $N$. (See \cite{JMOS} and the sources contained therein).

For each prime $\ell$, Deligne proved that there exists a representation $\rho_{\ell}:\mathrm{Gal}(\bar{\Q}/\Q)\to\mathrm{GL}_2(\Z_{\ell})$ with the property that if $p$ is a prime not dividing $\ell N$ and $\mathrm{Frob}_p$ is the Frobenius automorphism of $\mathrm{Gal}(\bar{\Q}/\Q)$ at $p$, then the characteristic polynomial of $\rho_{d}(\mathrm{Frob}_p)$ is $x^2-a_f(p)+p^{k-1}$.  By Deligne's proof of the Weil Conjectures (which establishes Part 2 of Hypothesis \ref{hyp}), we know that $|a_f(p)|\leq 2p^{(k-1)/2}$.  Thus the roots of the characteristic polynomial are $\alpha_p p^{(k-1)/2}$ and $\beta_p p^{(k-1)/2}$, where $\beta_p=\bar{\alpha}_p$ and $\alpha_p\beta_p=1$.  We recast the Euler product of $L(s,f)$ in \eqref{eqn:L(s,f)} as
\[
L(s,f)=\Big(\prod_{p\mid N }\frac{1}{1-a_f(p)p^{-s}}\Big)\prod_{p\nmid N }\prod_{j=0}^1\frac{1}{1-\alpha_p^j\beta_p^{1-j}p^{(k -1)/2-s}},
\]

When $N$ is squarefree, the Euler product of the $n$-th symmetric power of $f$, which we denote by $\mathrm{Sym}^n f$, is given by
\[
L(s,\mathrm{Sym}^n f)=\Big(\prod_{p\mid N }\frac{1}{1-a_f(p)^n p^{-s}}\Big)\prod_{p\nmid N }\prod_{j=0}^n\frac{1}{1-\alpha_p^j \beta_p^{n-j}p^{n(k -1)/2-s}}.
\]
(See Cogdell and Michel \cite[Section 1.1]{CM}.)  This is the $L$-function attached to the $\ell$-adic realizations of $\M=\mathrm{Sym}^n H^1(f)$; note that $L(s,\mathrm{Sym}^0 f)=\zeta(s)$ and $L(s,\mathrm{Sym}^1 f)=L(s,f)$.  The symmetric power $L$-functions of newforms determine the distribution of $a_f(p)/(2p^{(k-1)/2})$ in $[-1,1]$, but very little is known about their analytic properties (cf. \cite{sato-tate,Mazur}, for example).  Their critical values are important in the context of the Bloch-Kato conjecture, much like those of $L(s,f)$.  (See \cite{Watkins} for an accessible overview along with some convincing computations.)  The weight of $\mathrm{Sym}^n f$ is $n(k-1)$, the rank is $n+1$, and the global conductor is $N^n$.  (It is for this reason, and this reason alone, that we restrict $N$ to be squarefree.)    When $n=2r+1$ is odd, the integers which are critical for $\mathrm{Sym}^{2r+1}f$ are $r(k -1)+j$ for $1\leq j\leq k -1$.  The Hodge numbers all lie in $\{0,1\}$; see \cite{CM} for an exact expression for $L_{\infty}(s,\mathrm{Sym}^n f)$. From this we can check that the conditions of Theorem \ref{thm:main_theorem_1} (1) or (2) are satisfied under the assumptions of Corollary \ref{cor:sym^n}.

Conjecturally, we have $\mathrm{Sym}^n f\in\mathcal{A}_{n+1}(\Q)$ for each $n\geq 0$, and $\mathrm{Sym}^n f\in\mathcal{A}_{n+1}^{\mathrm{s}}(\Q)$ for each odd $n\geq1$.  Unconditionally, we know that $\mathrm{Sym}^n f\in\mathcal{A}_{n+1}(\Q)$ for each $n\leq 8$ (see Clozel and Thorne \cite{CT}, \cite{CM}, and the sources contained therein).  Moreover, as part of the celebrated proof of the Sato-Tate conjecture \cite{sato-tate}, we know that $L(s,\mathrm{Sym}^{n}f)$ can be analytically continued to the line $\Re(s)=1$ for each $n\geq1$.  It follows from the Euler product representation of $L(s,\mathrm{Sym}^n f)$ and \eqref{eqn:ext} that if $n\geq1$ is odd, then
\[
L(s,\mathrm{Ext}^2(\mathrm{Sym}^n f))=\zeta(s)\prod_{j=1}^{\frac{n-1}{2}}L(s,\mathrm{Sym}^{4j}f).
\]
In particular, if $n$ is odd and $\mathrm{Sym}^{n}f\in\mathcal{A}_{n+1}(\Q)$, then $L(s,\mathrm{Ext}^2(\mathrm{Sym}^n f))$ has a pole at $s=1$.  Thus by Lapid and Rallis \cite{LR}, we have that $\Lambda(\frac{n(k-1)+1}{2},f)\geq0$.  For $n=1$ and $n=3$, these results were proved by Waldspurger \cite{Waldspurger} and Kim \cite{Kim}, respectively.  Regardless of whether $N$ is squarefree, we expect that $L(s,\mathrm{Ext}^2(\mathrm{Sym}^n f))$ has a pole at $s=1$ for all odd $n\geq 1$, in which case $\mathrm{Sym}^n f\in\mathcal{A}_{n+1}^s(\Q)$ and we obtain the desired nonvanishing at the central critical point.

\subsection{Proof of Theorem \ref{thm:sym^n}}
By the modularity theorem, if $E$ is a semistable elliptic curve of squarefree conductor $N$, then $E$ corresponds to a weight 2 newform of level $N$, trivial nebentypus, and integral Fourier coefficients.  Thus $L(s,\mathrm{Sym}^n E)=L(s,\mathrm{Sym}^n f)$. By Corollary \ref{cor:sym^n}, the only cases left to check are 
\[
n=5, ~11\leq N\leq 43
\]
and 
\[
n=7,~11\leq N\leq 15.
\]
We observe that in all of these exceptional cases except for $(n,N)\in\{(5,37), (5,43)\}$, corresponding to the isogeny classes 37.a and 43.a in Cremona's table, the root number $\varepsilon(\mathrm{Sym}^n f)$ is $-1$; these are stored on the $L$-function and Modular Form Database (LMFDB) website at \url{http://www.lmfdb.org}. 

In the cases with $\varepsilon(\mathrm{Sym}^n f)=1$ (resp. $n=7$), we explicitly compute the zeros of $P_{\mathrm{Sym}^5 f}$ (resp. $P_{\mathrm{Sym}^7 f}$) and observe that all of them lie in the open unit disc. For this, we use the critical value $L(3,\mathrm{Sym}^5 f)$ and the Dirichlet coefficients of $L(s,\mathrm{Sym}^5 f)$ (resp. $L(s,\mathrm{Sym}^7 f)$), which are stored in the Lcalc files on \url{http://www.lmfdb.org}.

If $n=5$ and $\varepsilon(\mathrm{Sym}^5 f)=-1$, we have 
\[
P_{\mathrm{Sym}^5 f}(z)=\Lambda(5,\mathrm{Sym}^5 f)z^2+24\Lambda(4,\mathrm{Sym}^5 f)z,
\] 
so $P_{\mathrm{Sym}^5 f}$ has all zeros inside the unit disc, if 
\[
\Big|\frac{24\Lambda(4,\mathrm{Sym}^5 f)}{\Lambda(5,\mathrm{Sym}^5 f)}\Big|\le 1.
\]
This can again be checked by computing $L(4,\mathrm{Sym}^5 f)$ and $L(5,\mathrm{Sym}^5 f)$ in these cases. 

\section{Proof of Theorem \ref{thm:zeta}}
\label{sec:zeta}

We first present some corollaries of the results in \cite{RV}.  Let $U(z)$ be a polynomial of degree $e$ with $U(1)\neq0$.  Consider the rational function $V(z):=U(z)(1-z)^{-(e+1)}$.  It is easily shown that there exists a polynomial $H(z)$ of degree $e$ such that $H(\ell)=\frac{1}{
\ell!}\frac{d^{\ell}}{dz^{\ell}}V(z)\big|_{z=0}$ for each integer $\ell\geq0$.  Define $Z(s):=H(-s)$.
\begin{theorem}[Rodriguez-Villegas]
\label{thm:RV}
If all of the roots of $U$ lie on $\mathbb{S}^1$, then all of the roots of $Z(s)$ lie on the line $\Re(s)=1/2$.  Moreover, if $U$ has real coefficients and $U(1)\neq0$, then $Z(s)$ satisfies the functional equation $Z(1-s)=(-1)^e Z(s)$.	
\end{theorem}

 We now show that under the hypotheses of Theorem \ref{thm:sym^n}, $p_{\mathrm{Sym}^n E}(z)$ satisfies the hypotheses of Theorem \ref{thm:RV}.

\begin{lemma}
\label{lem:periods}
	Let $E/\Q$ be a semistable elliptic curve, and suppose that $\mathrm{Sym}^n E$ satisfies the hypotheses of Theorem \ref{thm:sym^n}.  If $\varepsilon(\mathrm{Sym}^n E)=1$, then $p_{\mathrm{Sym}^n E}(1)\neq0$.  If $\varepsilon(\mathrm{Sym}^n E)=-1$, then $p_{\mathrm{Sym}^n E}(z)$ has a simple zero at $z=1$.
\end{lemma}
\begin{proof}
	Let $n\geq3$ be odd, let $m=\frac{n-1}{2}$, and let $\varepsilon=\varepsilon(\mathrm{Sym}^n E)$.  By \eqref{eqn:p_def_func} and the fact that $L(s,\mathrm{Sym}^n E)$ is self-dual, we have that $p_{\mathrm{Sym}^n E}(1)$ equals
	\[
	\Big[\prod_{\nu=0}^m{2m-\nu\choose m}^{h_{\nu}}\Big]\Lambda(m+1,\mathrm{Sym}^n E)+2\sum_{j=1}^{m}\Big[\prod_{\nu=0}^m{2m-\nu\choose m-j}^{h_{\nu}}\Big]\Lambda(m+1+j,\mathrm{Sym}^n E)
	\]
	if $\varepsilon=1$ and $p_{\mathrm{Sym}^n E}(1)=0$ if $\varepsilon=-1$.
	
	When $\varepsilon=1$, it follows from Lemma \ref{lem:monotonicity} and Hypothesis \ref{hyp} (both of which hold whenever $\mathrm{Sym}^n E$ satisfies the hypotheses of Theorem \ref{thm:sym^n}) that the sum defining $p_{\mathrm{Sym}^n E}(1)$ has only nonnegative terms.  If $p_{\mathrm{Sym}^n E}(1)=0$, then it would follow that all Deligne periods of $\mathrm{Sym}^n E$ would equal zero.  This implies that the Deligne periods of $E$ are both zero, which is not true. (For the relationship between the periods of $E$ and the periods of $\mathrm{Sym}^n E$, see \cite{Watkins}, for example.)  Thus $p_{\mathrm{Sym}^n E}(1)\neq0$.
	
	Now, suppose that $\varepsilon=-1$.  Note that the sum defining $p'_{\mathrm{Sym}^n E}(z)$ is a sum of nonpositive terms.  Much like the case where $\varepsilon=1$, if all of these terms equal zero simultaneously, then all of the Deligne periods of $E$ are zero, which cannot happen.  Thus $p_{\mathrm{Sym}^n E}(z)$ has a simple zero at $z=1$.
\end{proof}

Define $\mathfrak{s}(m,n)$ by $\prod_{j=0}^{n}(x-j)=\sum_{m=0}^{n}\mathfrak{s}(n,m)x^{m}$.  Let
\[
\mathfrak{M}_{\mathrm{Sym}^n E}(j):=\frac{1}{(n-1)!}\sum_{m=0}^{n-1}\Big[\prod_{\nu=0}^{\frac{n-1}{2}} {n-1-\nu\choose \frac{n-1}{2}-|\frac{n-1-2m}{2}|}^{h_\nu}\Big]\Lambda(m+1,\mathrm{Sym}^n E)m^{j}
\]
and
\begin{equation}
\label{eqn:Z}
Z_{\mathrm{Sym}^n E}(s):=\varepsilon\sum_{h=0}^{n-1}(-s)^h\sum_{j=0}^{n-1-h}{h+j\choose h}\mathfrak{s}(n-1,h+j)\mathfrak{M}_{\mathrm{Sym}^n E}(j).
\end{equation}

\begin{proof}[Proof of Theorem \ref{thm:zeta}]
If $n\geq1$ is an integer, then we have the Maclaurin expansion
\[
(1-z)^{-n}=\sum_{\ell=0}^{\infty}{n-1+\ell\choose n-1}z^{\ell}.
\]
Sending $j$ to $n-1-j$ in the sum defining $p_{\mathrm{Sym}^n E}(z)$, using the functional equation for $\Lambda(s,\mathrm{Sym}^n E)$, and sending $\ell$ to $\ell+j-(n-1)$ yields the identity
\begin{equation}
\label{eqn:hl_def}
\frac{p_{\mathrm{Sym}^n E}(z)}{(1-z)^{n}}=\varepsilon\sum_{\ell=0}^{\infty}z^{\ell}\Big(\sum_{j=0}^{n-1}\Big[\prod_{\nu=0}^{\frac{n-1}{2}} {n-1-\nu\choose \frac{n-1}{2}-|\frac{n-1-2j}{2}|}^{h_\nu}\Big]\Lambda(j+1,\mathrm{Sym}^n E){\ell+j \choose n-1}\Big).
\end{equation}
Let $h_{\ell}$ be the coefficient of $z^{\ell}$ in \eqref{eqn:hl_def}.  With $\mathfrak{s}(n-1,m)$ defined above, we have
\[
h_{\ell}=\frac{\varepsilon}{(n-1)!}\sum_{h=0}^{n-1}\Big[\prod_{\nu=0}^{\frac{n-1}{2}} {n-1-\nu\choose \frac{n-1}{2}-|\frac{n-1}{2}-j|}^{h_\nu}\Big]\Lambda(j+1,\mathrm{Sym}^n E)\sum_{m=0}^{n-1}\mathfrak{s}(n-1,m)(\ell+j)^m
\]
which equals $Z_{\mathrm{Sym}^n E}(-\ell)$  (see \cite{ORS} for a similar manipulation).  This proves Part 3.

Let
\[
\hat{p}_{\mathrm{Sym}^n E}(z)=
\frac{p_{\mathrm{Sym}^n E}(z)}{(1-z)^{-\delta_{-1,\varepsilon}}},
\]
where $\delta_{i,j}$ is the Kronecker delta function.  By Theorem \ref{thm:sym^n} and Lemma \ref{lem:periods}, we see that $\hat{p}_{\mathrm{Sym}^n E}(z)$ is a polynomial of degree $n-1-\delta_{-1,\varepsilon}$, all of whose roots lie on $\mathbb{S}^1$.  Moreover, $\hat{p}_{\mathrm{Sym}^n E}(1)\neq0$.  Thus
\[
\frac{p_{\mathrm{Sym}^n E}(z)}{(1-z)^n}=\frac{\hat{p}_{\mathrm{Sym}^n E}(z)}{(1-z)^{n-\delta_{-1,\varepsilon}}}.
\]
Parts 1 and 2 follow from an application of Part 3 and Theorem \ref{thm:RV} with $e=n-1-\delta_{-1,\varepsilon}$.
\end{proof}

\def\cprime{$'$}

\end{document}